\documentclass{amsart}
\usepackage{hyperref}
\usepackage[capitalise]{cleveref}
\usepackage{thmtools}
\hypersetup{hypertexnames = false, bookmarksdepth = 2, bookmarksopen = true, colorlinks, linkcolor = black, citecolor = black, urlcolor = black, pdfstartview={XYZ null null 1}}

\usepackage{amsfonts}
\usepackage{amssymb}
\usepackage{booktabs}
\usepackage{diagbox}
\usepackage{enumitem}
\usepackage{mathtools}
\usepackage{parskip}
\usepackage{tikz-cd}
\usepackage[colorinlistoftodos, textsize = footnotesize]{todonotes}
\usepackage{xparse}
\usepackage{xspace}
\usepackage[T1]{fontenc}
\usepackage{mathrsfs}
\usepackage{eucal}
\usepackage{microtype}
\let\cal\mathcal
\def\Ascr{{\cal A}}

\def\Escr{{\cal E}}
\def\Fscr{{\cal F}}

\def\Hscr{{\cal H}}

\def\Kscr{{\cal K}}

\def\Nscr{{\cal N}}
\def\Oscr{{\cal O}}

\def\Qscr{{\cal Q}}
\def\Rscr{{\cal R}}

\def\Tscr{{\cal T}}
\def\Uscr{{\cal U}}

\let\blb\mathbb

\def\GG{{\blb G}}
\def \PP{{\blb P}}

\def \ZZ{{\blb Z}}

\def \HH{{\blb H}}

\def\id{\text{id}}

\def\Gr{\operatorname{Gr}}

\def\QGr{\operatorname{QGr}}
\DeclareMathOperator\qgr{qgr}

\def\gr{\operatorname{gr}}

\def\gr{\operatorname {gr}}

\def\Ext{\operatorname {Ext}}

\def\End{\operatorname {End}}
\def\RHom{\operatorname {RHom}}

\def\coker{\operatorname {coker}}
\def\ker{\operatorname {ker}}
\def\Ker{\operatorname {ker}}

\def\End{\operatorname {End}}

\def\id{{\operatorname {id}}}

\DeclareMathOperator{\PProj}{\mathbf{Proj}}

\DeclareMathOperator{\eval}{ev}

\declaretheorem[name=Theorem,numberwithin=section]{theorem}
\declaretheorem[sibling=theorem]{question}
\declaretheorem[name=Corollary,sibling=theorem]{corollary}
\declaretheorem[name=Definition,sibling=theorem]{definition}
\declaretheorem[name=Example,sibling=theorem]{example}
\declaretheorem[name=Lemma,sibling=theorem]{lemma}
\declaretheorem[name=Proposition,sibling=theorem]{proposition}
\declaretheorem[name=Remark,sibling=theorem]{remark}

\DeclareMathOperator\Hom{Hom}
\DeclareMathOperator\identity{id}

\def\opp{\operatorname{op}}

\def\Hilb{\operatorname{Hilb}}
\DeclareMathOperator\coh{coh}
\def\Bl{\operatorname{Bl}}
\def\Sym{\operatorname{Sym}}

\def\Sch{\operatorname{Sch}}
\def\Sets{\operatorname{Sets}}
\def\HHH{\operatorname{H}}

\newcommand\bounded{\ensuremath{\mathrm{b}}}

\newcommand\LLL{\ensuremath{\mathbf{L}}}
\newcommand\quadric{\ensuremath{\mathbb{P}^1\times\mathbb{P}^1}}

\newcommand\LL{\ensuremath{\mathrm{L}}}

\DeclareMathOperator\derived{\mathbf{D}}
\DeclareMathOperator\Grass{Gr}
\DeclareMathOperator\HHHH{HH}

\mathchardef\mhyphen="2D
\newcommand\dash{\nobreakdash-\hspace{0pt}}

\begin{document}

\title[Noncommutative quadrics and Hilbert squares]{Derived categories of noncommutative quadrics and Hilbert squares}
\author{Pieter Belmans}
\email[Pieter Belmans]{pbelmans@mpim-bonn.mpg.de}
\address{Max Planck Institute for Mathematics, Vivatsgasse 7, 53111 Bonn, Germany}
\author{Theo Raedschelders}
\email[Theo Raedschelders]{Theo.Raedschelders@glasgow.ac.uk}
\address{University of Glasgow, University Place, Glasgow, G12 8SQ, United Kingdom}

\begin{abstract}
A non-commutative deformation of a quadric surface is usually described by a three-dimensional cubic Artin--Schelter regular algebra. In this paper we show that for such an algebra its bounded derived category embeds into the bounded derived category of a commutative deformation of the Hilbert scheme of two points on the quadric. This is the second example in support of a conjecture by Orlov. Based on this example we formulate an infinitesimal version of the conjecture, and provide some evidence in the case of smooth projective surfaces.
\end{abstract}

\maketitle



\section{Introduction}
\label{section:introduction}
In this paper we study the derived category of a quadric (and its noncommutative analogues) in relationship with the derived category of the Hilbert scheme of two points on a quadric (and commutative deformations thereof). The motivation comes from several seemingly disparate observations.

First of all let~$S$ be a smooth projective surface (over an algebraically closed field~$k$ throughout). Then it is a classical result of Fogarty that the Hilbert scheme of~$n$ points~$\Hilb^nS$ is again a smooth projective variety, of dimension~$2n$ \cite{MR0237496}. If we moreover assume that~$\HHH^1(S,\mathcal{O}_S)=\HHH^2(S,\mathcal{O}_S)=0$ (e.g.~$S$ is a rational surface, such as a quadric) then \cite{MR3397451} shows that the Fourier--Mukai functor
\begin{equation}
  \label{equation:krug-sosna-embedding}
  \Phi_{\mathcal{I}_{\mathcal{U}_n}}\colon\derived^\bounded(S)\to\derived^\bounded(\Hilb^n S)
\end{equation}
is a fully faithful functor for~$n\geq 2$, where~$\mathcal{I}_{\mathcal{U}_n}$ is the ideal sheaf for the universal family~$\mathcal{U}_n\subset S\times\Hilb^nS$.

Another piece of motivation stems from the notion of geometric dg~categories as introduced in \cite{orlov-smooth-and-proper}. He shows that any dg~category whose homotopy category has a full exceptional collection can be embedded in (an enhancement of) the derived category of a smooth projective variety. This construction can be applied to the full exceptional collection describing the derived category of a quadric surface, but the resulting variety is constructed using iterated projective bundles and does not seem to have a geometric interpretation in terms of a moduli problem, unlike the embedding \eqref{equation:krug-sosna-embedding}.

It is interesting to try and apply Orlov's algorithm to a dg~category with a full exceptional collection which is not of geometric origin. This brings us to the final piece of motivation: deformations of abelian categories. In noncommutative algebraic geometry a central role is played by abelian categories and their derived categories, and there is a framework for describing the deformations of abelian categories \cite{MR2183254,MR2238922}, so in particular it can be applied to~$\coh\quadric$. The quadric is easily seen to be rigid (i.e.~$\HHH^1(\quadric,\Tscr_{\quadric})=0$), but its category of coherent sheaves has nontrivial deformations ($\HHHH^2(\coh\quadric)\cong\HHH^0(\quadric,\bigwedge^2\Tscr_{\quadric})$ is~9\dash dimensional), which can be seen as deformations of~$\quadric$ in the noncommutative direction.

Because the quadric has a strongly ample sequence it is moreover possible to pass from infinitesimal deformations to formal deformations \cite{MR3050709}, and the theory has been worked out in detail in \cite{MR3108697,MR2836401}. A noncommutative quadric is an abelian category~$\qgr A$, which is a certain quotient category of the category of graded modules for a (generalized) graded algebra~$A$ satisfying some natural conditions. On the derived level it is possible to view a family of noncommutative quadrics by varying the relations in the quiver coming from the full and strong exceptional collection \cite{okawa-ueda-nc-quadrics}. For these new exceptional collections it makes sense to apply Orlov's embedding result, but again the result is an iterated projective bundle construction where arbitrary choices have been made and there is no moduli interpretation.

Yet for~$\mathbb{P}^2$ (and its noncommutative deformations) Orlov shows that there exists an embedding in (a commutative deformation of)~$\Hilb^2\mathbb{P}^2$ \cite{zbMATH06527075}, hence there \emph{is} a moduli interpretation for the derived category of the finite-dimensional algebras whose structure resembles that of the Be{\u\i}linson quiver for~$\mathbb{P}^2$.

In this paper we obtain a result analogous to Orlov's for noncommutative quadrics. The following is a compressed version of \cref{theorem:noncommutative-embedding} and is our main result.
\begin{theorem}
  \label{theorem:main-theorem-compressed}
  For a generic noncommutative quadric~$A$ there exists a deformation~$H$ of~$\Hilb^2\quadric$ and a fully faithful embedding
  \begin{equation}
    \derived^\bounded(\qgr A)\hookrightarrow\derived^\bounded(\coh H).
  \end{equation}
\end{theorem}
Recall that a fully faithful embedding is automatically admissible in this context.

To prove this result we need an explicit geometric model for~$\Hilb^2\quadric$, which we give in \cref{theorem:description-H}. In \cref{subsection:geometric-squares} we explain how this geometric model depends on so called geometric squares: linear algebra data that describes the composition law in the derived category. In \cref{subsection:noncommutative-quadrics} it is shown how a sufficiently generic noncommutative quadric gives rise to such a geometric square, and indeed to an embedding as in \cref{theorem:main-theorem-compressed}.

Note that there exists a notion of Hilbert scheme of points for a general cubic Artin--Schelter regular graded algebra \cite{MR2309895}, which is a subset of all noncommutative quadrics. We do not address the comparison between these moduli spaces and the deformations constructed in this paper.

We also formulate a general question regarding limited functoriality of Hochschild cohomology and the Hochschild--Kostant--Rosenberg decomposition, motivated by a conjecture of Orlov. This is done in \cref{section:remarks}. We discuss some evidence suggesting an interesting relationship between the Hochschild cohomology of a surface and the Hochschild cohomology of the Hilbert scheme of points, showing that the results in this paper hint towards a much more general picture.

\textbf{Acknowledgements}
We would like to thank Michel Van den Bergh for suggesting the exceptional collection in~\cref{theorem:exceptional-collection} and for many other interesting discussions. We would like to thank the referee for interesting suggestions and corrections, in particular for the more natural proof of Lemma \ref{lemma:factorization}. The first author would also like to thank Shinnosuke Okawa for interesting discussions.

The first author is supported by the Max Planck Institute for Mathematics in Bonn, and was at the time of writing supported by a Ph.D.~fellowship of the Research Foundation--Flanders (FWO). The second author is supported by an EPSRC postdoctoral fellowship EP/R005214/1, and was at the time of writing supported by a Ph.D.~fellowship of the Research Foundation--Flanders (FWO).

\section{The geometry of \texorpdfstring{$\Grass(1,3)$}{Gr(1,3)} and \texorpdfstring{$\Hilb^2(\quadric)$}{Hilb(2,S)}}
\label{section:description-G-and-H}

Throughout the paper, we will assume~$k$ is an algebraically closed field of characteristic~$0$.

\subsection{Grassmannians}
\label{subsection:grassmannians}
Let~$V$ be a vector space of dimension~$n$ and~$l$ an integer with~$n\ge l+1$. We let~$\GG$ be the Grassmannian of~$l$\dash dimensional quotients of~$V$, and closed points will be denoted with square brackets, for example~$[V \twoheadrightarrow W]$. Let
\begin{equation}
  0 \to \Rscr \xrightarrow{r} V\otimes_k \Oscr_{\GG} \xrightarrow{q} \Qscr \to 0
\end{equation}
be the tautological exact sequence on~$\GG$, where~$\Qscr$ is the universal quotient bundle of rank~$l$, and~$\Rscr$ is the universal subbundle of rank~$n-l$. Also, put~$\Oscr_{\GG}(1)=\bigwedge\nolimits^l\Qscr$.

The Grassmannian is a fine moduli space for the functor~$F_{\GG}\colon\Sch^{\opp} \to \Sets$ sending a scheme~$X$ to the set of epimorphisms~$V \otimes_k \Oscr_X \twoheadrightarrow \Fscr$, where~$\Fscr$ is a rank~$l$ vector bundle on~$X$. Hence, there is a bijection
\begin{equation}
  \label{equation:bijection-universal-property}
  \Hom_{\Sch}(X,\GG) \to F_{\GG}(X): \Phi \mapsto \Phi^*q
\end{equation}
In particular, if~$\Phi$ is an closed immersion, then~$\Fscr$ is just the restriction of~$\Qscr$ to~$X$. The inverse of \eqref{equation:bijection-universal-property} is constructed as follows: given an epimorphism~$\phi: V\otimes_k \Oscr_X \twoheadrightarrow \Fscr$, we define an element~$\Phi \in \Hom_{\Sch}(X,\GG)$ by
\begin{equation}
  \Phi(x)=[V \xrightarrow{\phi_x} \Fscr_x \otimes_{\Oscr_x} k(x)].
\end{equation}

From now on we will focus on a specific case. Suppose~$\dim_kV=4$ and~$l=2$. Let~$V_0, V_1$ denote vector spaces of dimension~$2$, and suppose~$\chi\colon V \to V_0 \otimes_k V_1$ is a given isomorphism. Consider the epimorphism~$\phi_\chi$ defined by the composition
\begin{equation}
  \label{equation:epimorphism}
  \begin{tikzcd}
    V\otimes_k\Oscr_{\PP(V_0)} \arrow[d, "\chi\otimes_k\identity"] \\
    V_0\otimes_k V_1\otimes_k\Oscr_{\PP(V_0)} \arrow[d, "\cong"] \\
    \HHH^0(\PP(V_0),\Oscr_{\PP(V_0)},V_1\otimes_k\Oscr_{\PP(V_0)}(1))\otimes_k\Oscr_{\PP(V_0)} \arrow[d, "\eval"] \\
    V_1\otimes_k\Oscr_{\PP(V_0)}(1).
  \end{tikzcd}
\end{equation}
Then under \eqref{equation:bijection-universal-property}, $\phi_{\chi}$ induces a closed immersion
\begin{equation}
  \Phi_{\chi}\colon\PP(V_0) \hookrightarrow \GG:p \mapsto [V \xrightarrow{\chi} V_0 \otimes_k V_1 \xrightarrow{p \otimes_k \identity_{V_1}} V_1]
\end{equation}

The following lemma will be used in \cref{subsection:geometric-squares} to construct strong exceptional collections.

\begin{lemma}
  \label{lemma:factorization}
  Any decomposition $\chi:V \to V_0 \otimes_k V_1$ gives rise to an isomorphism
  \begin{equation}
    \begin{aligned}
      F_{\chi}\colon\Hom_{\Oscr_{\GG}}(\Rscr,\Kscr_{\chi}) \otimes_k \Hom_{\Oscr_{\GG}}(\Kscr_{\chi},\Oscr_{\GG}) & \to \Hom_{\Oscr_{\GG}}(\Rscr,\Oscr_{\GG}) \\
      f \otimes_k g & \mapsto g \circ f,
    \end{aligned}
  \end{equation}
  where $\Kscr_{\chi}$ is the coherent sheaf
  \begin{equation}
    \Kscr_{\chi}\coloneqq\Ker\big(\HHH^0(\GG, \Oscr_{\PP(V_0)}(1)) \otimes_k \Oscr_{\GG} \xrightarrow{\eval} \Oscr_{\PP(V_0)}(1)\big),
  \end{equation}
  and $\PP(V_0)$ is embedded into $\GG$ via $\Phi_{\chi}$.
\end{lemma}

The proof of this lemma is quite technical and is relegated to the appendix.

\subsection{Hilbert schemes of points}
\label{subsection:hilbert-scheme-of-points}
The Hilbert scheme is a classical object in algebraic geometry, parametrising closed subschemes of a projective scheme. One can associate a Hilbert polynomial to a closed subscheme, and this gives rise to a disjoint union decomposition of the Hilbert scheme. In particular, for the constant Hilbert polynomial we get the \emph{Hilbert scheme of points}.

For a smooth projective curve~$C$ one has that~$\Hilb^nC=\Sym^nC$. In particular it is again smooth projective and of dimension~$n$. For a smooth projective surface~$S$ it can be shown that~$\Hilb^nS$ is again smooth projective and of dimension~$2n$. For higher-dimensional varieties and~$n\gg 2$ the Hilbert scheme becomes (very) singular.

We will identify~$\quadric$ with its image under the Segre embedding
\begin{equation}
  \label{equation:segre}
  \quadric \hookrightarrow \PP^3:([x_0:x_1],[y_0:y_1]) \mapsto [x_0y_0:x_0y_1:x_1y_0:x_1y_1],
\end{equation}
which we denote by~$Q$, a smooth quadric surface. This surface has two rulings, and every line on~$Q$ defines a point of~$\GG$. We denote
\begin{equation}
  \label{equation:definition-L}
  L\coloneqq L_0 \sqcup L_1=\{l \in \GG\mid l \subset Q\} \subset \GG,
\end{equation}
where~$L_0$ (respectively~$L_1$) corresponds to the lines in the first (respectively second) ruling. Note that~$L_0 \cap L_1=\emptyset$, and each of these two lines determines a factorization of~$V$ as in~\cref{lemma:factorization}.

The following proposition provides our main model for working with the Hilbert scheme~$\HH\coloneqq\Hilb^2(\quadric)$ of two points on~$\PP^1 \times \PP^1$. A reference for this description is \cite[Theorem~1.1]{MR2327036}.
\begin{proposition}
  \label{theorem:description-H}
  There is an isomorphism~$\HH \cong \Bl_L\GG$.
\end{proposition}
\begin{proof}
  Using the Segre embedding~\eqref{equation:segre} there exists a surjective morphism
  \begin{equation}
  \label{equation:blowup-morphism}
    f\colon\HH\to\GG\colon [Z]\mapsto l_{[Z]}
  \end{equation}
  where the line~$l_{[Z]}$ for a point~$[Z]\in\HH$ is defined to be the line through the two points if~$[Z]$ corresponds to two distinct points, otherwise we use the tangent vector to define the line.

  On the open set~$\GG\setminus L$ this is a bijection, the inverse being given by the morphism mapping a line in~$\mathbb{P}^3$ to its intersection with the quadric.

  On the closed set~$L\subseteq\GG$ the fiber over~$l\in L$ can be identified with~$\mathbb{P}^2$: it is formed by the set of pairs of points on the line~$l$, hence~$\HH_l\cong\Sym^2\mathbb{P}^1\cong\mathbb{P}^2$.

  By the uniqueness property of blow-ups, see e.g.~\cite[\S4.6]{MR1288523}, we get the proposed isomorphism.

  A more abstract description of the morphism \eqref{equation:blowup-morphism} can be obtained using the moduli interpretation \eqref{equation:bijection-universal-property} of the Grassmannian. Denoting by
  \begin{equation}
    \Phi=\{(x,\xi) \in Q \times \mathbb{H} \mid x \in \xi\} \cong \Bl_{\Delta}(Q \times Q)
  \end{equation}
  the universal family for $\HH$, with projection morphisms $\text{pr}_1:\Phi \to Q$ and $\text{pr}_2:\Phi \to \HH$, we define $\Escr:=\text{pr}_{2*}\text{pr}_1^*\Oscr(1,1)$. This is a vector bundle of rank $2$, and one checks that
  \begin{equation}
    \label{equation:identification-sections}
    \HHH^0(\HH,\Escr) \cong \HHH^0(\Phi,\text{pr}_1^*\Oscr(1,1)) \cong \HHH^0(Q,\Oscr(1,1)).
  \end{equation}
  By pushing forward the evaluation morphism
  \begin{equation}
    \HHH^0(\Phi,\text{pr}_1^*\Oscr(1,1)) \otimes_k \Oscr_{\Phi} \to \text{pr}_1^*\Oscr(1,1)
  \end{equation}
  along $\text{pr}_2$, and using \eqref{equation:identification-sections}, one obtains the evaluation morphism
  \begin{equation}
    \label{equation:evaluation-morphism}
    \HHH^0(Q,\Oscr(1,1)) \otimes_k \Oscr_{\HH} \to \Escr,
  \end{equation}
  which on the fiber over $\xi \in \HH$ is just the obvious restriction
  \begin{equation}
    \HHH^0(\Oscr(1,1)) \to \HHH^0(\Oscr(1,1)|_{\xi}).
  \end{equation}
  Note that these restrictions are all surjective because $\Oscr(1,1)$ is very ample, so also \eqref{equation:evaluation-morphism} is surjective. Under the bijection \eqref{equation:bijection-universal-property}, \eqref{equation:evaluation-morphism} corresponds exactly to \eqref{equation:blowup-morphism}. For more on this construction and its relation to ($n$)-very ampleness of line bundles, see \cite{MR1065935}.
\end{proof}

\begin{remark}
  \label{remark:description-H-orlov}
  In \cite{zbMATH06527075} the embedding of a noncommutative~$\mathbb{P}^2$ into the derived category of a deformation of~$\Hilb^2\mathbb{P}^2$ is based on the description of the Hilbert scheme as~$\Hilb^2\mathbb{P}^2\cong\mathbb{P}(\Sym^2\Tscr_{\mathbb{P}^2}(-1)^{\vee})$.
\end{remark}

To find an exceptional collection on~$\HH$ that is compatible with deformations we need to describe some bundles on~$\GG$ and on~$\HH$ more explicitly. Based on \cref{theorem:description-H} we will use the following notation for the rest of the paper.
\begin{equation}
  \label{equation:notation-commutative}
  \begin{tikzcd}[row sep = huge, column sep = large]
    E=E_0\sqcup E_1 \arrow[hook]{r}{j=j_0\sqcup j_1} \arrow[two heads]{d}{q=q_0\sqcup q_1} & \mathbb{H}\coloneqq\Bl_{L_0\sqcup L_1}\GG \arrow[two heads]{d}{p} \\
    L=L_0\sqcup L_1 \arrow[hook]{r}{i=i_0\sqcup i_1} & \GG.
  \end{tikzcd}
\end{equation}

\begin{lemma}
  \label{lemma:NLi-QLi}
  There are isomorphisms
  \begin{equation}
    \begin{aligned}
      \Qscr|_{L_i} &\cong \Oscr_{\PP^1}(1)^{\oplus 2}, \\
      \Rscr|_{L_i} & \cong \Oscr_{\PP^1}(-1)^{\oplus 2}, \\
      \Nscr_{L_i} \GG &\cong \Oscr_{\PP^1}(2)^{\oplus 3}.
    \end{aligned}
  \end{equation}
\end{lemma}
\begin{proof}
  The first two isomorphisms follows from~\eqref{equation:bijection-universal-property} since the~$L_i$ are embedded using an exact sequence
  \begin{equation}
    0 \to \Oscr_{L_i}(-1)^{\oplus 2} \to V \otimes_k \Oscr_{L_i} \to \Oscr_{L_i}(1)^{\oplus 2} \to 0
  \end{equation}
  as in~\eqref{equation:epimorphism}.

  For the third isomorphism, since the tangent bundle~$\Tscr_\GG$ can be expressed as~$\Hscr om(\Rscr,\Qscr)$, we get for the normal bundle
  \begin{equation}
    \begin{aligned}
      \Nscr_{L_i} \GG &= \coker(\Tscr_{\PP^1} \to \Hscr om(\Rscr,\Qscr)|_{\PP^1}) \\
      &= \coker(\Oscr_{\PP^1}(2) \to \Oscr_{\PP^1}(2)^{\oplus 4}) \\
      &=\Oscr_{\PP^1}(2)^{\oplus 3}.
    \end{aligned}
  \end{equation}
\end{proof}

From the description of the normal bundle~$\Nscr_{L_i}\GG$ in \cref{lemma:NLi-QLi} we find that
\begin{equation}
  E_i\cong\PProj_{\PP^1}\left( \Sym(\Oscr_{\PP^1}(2)^{\oplus3})^\vee \right)\cong\PP^1\times\PP^2.
\end{equation}
We will use the following notation
\begin{equation}
  \Oscr_{E_i}(m,n)\coloneqq \Oscr_{\PP^1}(m) \boxtimes \Oscr_{\PP^2}(n).
\end{equation}
Whenever we write~$\Oscr_E(m,n)$ this means that we use this construction for both connected components.

For the final lemma, recall that~$\Oscr_E(E)$ is shorthand for~$\Oscr_\HH(E)|_E=j^*(\Oscr_\HH(E))$, which can also be written as~$\Nscr_E\HH$. Using this notation we can describe~$\omega_\HH$ and two bundles on the exceptional locus~$E$ as follows.

\begin{lemma}
  \label{lemma:bundles-H-and-E}
  There are isomorphisms
  \begin{equation}
    \omega_{\HH} \cong p^*\left( \bigwedge\nolimits^2\Qscr \right)^{\otimes -4}(2E),
  \end{equation}
  and
  \begin{equation}
    \label{equation:restriction-canonical}
    \begin{aligned}
      \Oscr_{E}(E) &\cong \Oscr_{E}(2,-1), \\
      \omega_{\HH}|_{E} &\cong \Oscr_E(-4,-2).
    \end{aligned}
  \end{equation}
\end{lemma}
\begin{proof}
  Applying the adjunction formula and the isomorphism~$\Nscr_{L_i} \GG \cong \Oscr_{\PP^1}(2)^{\oplus 3}$ from \cref{lemma:NLi-QLi}, we find
  \begin{equation}
    \begin{aligned}
      \omega_{\PP^1} &\cong i^*(\omega_{\GG}) \otimes \det (\Nscr_{L_i}\GG) \\
      \Leftrightarrow \Oscr_{\PP^1}(-2) &\cong \omega_{\GG}|_{\PP^1} \otimes \Oscr_{\PP^1}(6) \\
      \Leftrightarrow \Oscr_{\PP^1}(-8) &\cong \omega_{\GG}|_{\PP^1}.
    \end{aligned}
  \end{equation}

  For the canonical bundles, we get
  \begin{equation}
    \label{equation:omega-H}
    \omega_{\HH}\cong p^*(\omega_{\GG}) \otimes \Oscr_\HH(2E), \\
  \end{equation}
  and
  \begin{equation}
    \label{equation:omega-E}
    \omega_{E}\cong(\omega_{\HH} \otimes \Oscr_{\HH}(E))|_{E}.
  \end{equation}
  Now plug~\eqref{equation:omega-H} into~\eqref{equation:omega-E} and use~$\omega_{E_i} \cong \Oscr_{E_i}(-2,-3)$ to get
  \begin{equation}
    \begin{aligned}
      \Oscr_E(-2,-3) & \cong \Oscr_E(-8,0) \otimes \Oscr_E(3E) \\
      \Leftrightarrow \Oscr_E(6,-3) &\cong \Oscr_E(3E) \\
      \Leftrightarrow \Oscr_E(E) &\cong \Oscr_E(2,-1).
    \end{aligned}
  \end{equation}
  Finally~\eqref{equation:omega-H} provides
  \begin{equation}
    \omega_{\HH}|_{E} \cong \Oscr_E(-8,0) \otimes \Oscr_E(4,-2) \cong \Oscr_E(-4,-2),
  \end{equation}
  completing the proof.
\end{proof}

\subsection{The derived category of \texorpdfstring{$\Grass(1,3)$}{Gr(1,3)} and Orlov's blow-up formula}
\label{subsection:derived-categories}

The following theorem is a particular case of a more general result obtained in~\cite{MR3371493,MR939472}.
\begin{theorem}
  \label{theorem:decomposition-G}
  The derived category of~$\GG$ has a full and strong exceptional collection
  \begin{equation}
    \derived^\bounded(\GG)=\left\langle
      \bigwedge\nolimits^2 \Rscr \otimes \bigwedge\nolimits^2 \Rscr,
      \bigwedge\nolimits^2 \Rscr \otimes \Rscr,
      \bigwedge\nolimits^2\Rscr,
      \Sym^2 \Rscr,
      \Rscr,
      \Oscr_{\GG}
    \right\rangle.
  \end{equation}
\end{theorem}

\begin{remark}
  In fact, we will only need the exceptional pair~$\langle \Rscr,\mathcal{O}_\GG\rangle$, which can also be established by elementary means.
\end{remark}

We know from \cref{theorem:description-H} that~$\HH \cong \Bl_L(\GG)$, so the following classical result of Orlov describes the derived category of~$\HH$. Let~$Y$ be a smooth subvariety of codimension~$r$ in a smooth algebraic variety~$X$. Then there exists a cartesian square
\begin{equation}
  \label{equation:cartesian-square-blowup}
  \begin{tikzcd}
    E_Y \arrow{r}{j} \arrow{d}{q} & \Bl_YX \arrow{d}{p} \\
    Y \arrow{r}{i} & X
  \end{tikzcd}
\end{equation}
where~$i$ and~$j$ are closed immersions, and~$q\colon E_Y \to Y$ is the projective bundle of the exceptional divisor~$E_Y$ in~$\Bl_YX$ on~$Y$.

\begin{theorem}\cite[Theorem~4.3]{MR1208153}
  \label{theorem:blowup}
  There is a semi-orthogonal decomposition
  \begin{equation}
    \derived^\bounded(\Bl_YX) = \langle \derived^\bounded(X), \derived^\bounded(Y)_0, \ldots, \derived^\bounded(Y)_{r-2} \rangle.
  \end{equation}
\end{theorem}

In this statement,~$\derived^\bounded(X)$ is the full subcategory of~$\derived^\bounded(\Bl_YX)$ which is the image of~$\derived^\bounded(X)$ under
\begin{equation}
  \mathbf{L}p^*\colon\derived^\bounded(X) \to \derived^\bounded(\Bl_YX),
\end{equation}
and~$\derived^\bounded(Y)_k$ is the full subcategory of~$\derived^\bounded(\Bl_YX)$ which is the image of~$\derived^\bounded(Y)$ under
\begin{equation}
  \label{equation:orlov-blowup-fully-faithful}
  \mathbf{R}j_*(\Oscr_{E_Y}(k) \otimes q^*(-)):\derived^\bounded(Y) \to \derived^\bounded(\Bl_YX).
\end{equation}

\begin{corollary}
  There is a semi-orthogonal decomposition
  \begin{equation}
    \label{equation:decomposition-Db(H)}
    \begin{aligned}
      \derived^\bounded(\HH)
      &=\left\langle \derived^\bounded(\GG),\derived^\bounded(L)_0,\derived^\bounded(L)_1 \right\rangle \\
      &=\left\langle \derived^\bounded(\GG),\derived^\bounded(L_0)_0,\derived^\bounded(L_0)_1,\derived^\bounded(L_1)_0,\derived^\bounded(L_1)_1 \right\rangle \\
    \end{aligned}
  \end{equation}
  In particular there exists a full exceptional collection of length~14\ in~$\derived^\bounded(\HH)$.
\end{corollary}

\begin{remark}
  \label{remark:krug-sosna-embedding}
  This is not the only way of obtaining a semi-orthogonal decomposition of the Hilbert scheme in this situation. For an arbitrary surface~$S$ one obtains using equivariant derived categories \cite{MR2584227} and the description of the Hilbert scheme of points as a quotient that there exists a full (and strong) exceptional collection in~$\derived^\bounded(\Hilb^nS)$ provided there exists a full (and strong) exceptional collection in~$\derived^\bounded(S)$ \cite[Proposition~1.3 and Remark~4.6]{MR3397451}.
\end{remark}

\section{Embedding derived categories of noncommutative quadrics}

\subsection{Geometric squares and deformations of \texorpdfstring{$\Hilb^2(\quadric)$}{Hilb(2,S)}}
\label{subsection:geometric-squares}
Recall~\cite{MR939472} that for the derived category of the quadric~$Q$ there is a full and strong exceptional collection
\begin{equation}
  \label{equation:block-collection}
  \begin{tikzcd}
    & \Oscr_Q(-1,0) \arrow[swap]{dr}{d_2} \arrow[shift left = .6em]{dr}{d_1} & \\
    \Oscr_Q(-1,-1) \arrow[swap]{ur}{c_2} \arrow[shift left = .6em]{ur}{c_1} \arrow[swap]{dr}{a_2} \arrow[shift left = .6em]{dr}{a_1} & & \Oscr_Q(0,0) \\
    & \Oscr_Q(0,-1) \arrow[swap]{ur}{b_2} \arrow[shift left = .6em]{ur}{b_1} &
  \end{tikzcd}
\end{equation}
with relations~$b_ia_j = d_jc_i$, for~$i,j \in \{1,2\}$. We isolate some of the properties of this exceptional collection in the following definition.
\begin{definition}
  A \emph{geometric square} is a septuple~$\square=(V,U_0^0, U_1^0, U_0^1, U_1^1, \phi_0, \phi_1)$, where~$V$ is a~4\dash dimensional vector space, the~$U_j^i$ are~2\dash dimensional vector spaces, and the~$\phi_i$ are isomorphisms
  \begin{equation}
    \phi_i\colon V\to U_0^i \otimes_k U_1^i.
  \end{equation}
\end{definition}

Using \cref{lemma:factorization}, the two isomorphisms~$\phi_i$ in a geometric square give rise to two embeddings~$L_i\coloneqq\PP(U_0^i) \hookrightarrow \GG$ and sheaves
\begin{equation}
  \label{equation:definition-Ki}
  \mathcal{K}_i\coloneqq\ker\left( \HHH^0(\GG,\Oscr_{L_i}(1))\otimes_k\Oscr_\GG\to\Oscr_{L_i}(1) \right).
\end{equation}

\begin{proposition}
  \label{proposition:relations}
  For a sufficiently generic geometric square~$\square$, the Ext-quiver of the endomorphism algebra
  \begin{equation}
    \label{equation:endomorphism-algebra}
    Q_{\square}\coloneqq\End_\GG(\Rscr \oplus \mathcal{K}_0 \oplus \mathcal{K}_1 \oplus \Oscr_\GG)
  \end{equation}
  is of the form~\eqref{equation:block-collection}, and moreover
  \begin{equation}
    \dim \Hom(\Rscr, \Oscr_\GG) = 4.
  \end{equation}
\end{proposition}

\begin{proof}
  We first check that there are no Hom's going backwards. Applying $\Hom(-,\mathcal{R})$ to \eqref{equation:definition-Ki} we see that by exceptionality of the pair~$\langle\Rscr,\Oscr_\GG\rangle$ we need to prove that~$\Ext^1(\Oscr_{L_i}(1),\Rscr)=0$. This is the case by Serre duality:
  \begin{equation}
    \begin{aligned}
      \Ext_\GG^1(\Oscr_{L_i}(1),\Rscr)
      &\cong\Ext_\GG^3(\Rscr\otimes\omega_\GG^\vee,\Oscr_{L_i}(1))^\vee \\
      &\cong\Ext_{L_i}^3((\Rscr\otimes\omega_\GG^\vee)|_{L_i},\Oscr_{L_i}(1))^\vee \\
      &=0.
    \end{aligned}
  \end{equation}

  Applying~$\Hom(\Oscr_\GG,-)$ to \eqref{equation:definition-Ki} we get that~$\mathcal{K}_i$ indeed does not have global sections because we get the identity morphism between~$\Hom(\Oscr_\GG,\HHH^0(\GG,\Oscr_{L_i}(1))\otimes_k\Oscr_\GG)$ and~$\Hom(\Oscr_\GG,\Oscr_{L_i}(1))$.

  Now to each of the isomorphisms~$\phi_i$ we can apply~\cref{lemma:factorization}, and for a generic geometric square, the~$\PP(U_0^i)$ don't intersect in~$\GG$, hence~$\Hom(\mathcal{K}_i,\mathcal{K}_{1-i})=0$, and the algebra~$Q_\square$ does indeed have the form~\eqref{equation:block-collection}.
\end{proof}

These four coherent sheaves cannot be used to realize an admissible embedding~$\derived^\bounded(Q_{\square}) \hookrightarrow \derived^{\bounded}(\GG)$ since they do not form an exceptional collection. To ensure that they do, we need to blow up~$\GG$ in the two~$L_i$, mimicking the description in \cref{theorem:description-H}. Let us denote by~$E_i$ the corresponding exceptional divisors on~$\HH_{\square}\coloneqq\Bl_{L_0 \sqcup L_1} \GG$, so we have a cartesian square
\begin{equation}
  \label{equation:notation-square}
  \begin{tikzcd}[row sep = huge, column sep = large]
    E_{\square}=E_0\sqcup E_1 \arrow[hook]{r}{j=j_1\sqcup j_2} \arrow[two heads]{d}{q=q_1\sqcup q_2} & \mathbb{H}_{\square}=\Bl_{L_{0}\sqcup L_1}\GG \arrow[two heads]{d}{p} \\
    L_{\square}=L_0\sqcup L_1 \arrow[hook]{r}{i=i_1\sqcup i_2} & \GG
  \end{tikzcd}
\end{equation}
similar to \eqref{equation:notation-commutative}.

We are now ready to show how a generic geometric square gives rise to a strong exceptional collection of vector bundles. In \cref{theorem:embedding} we will describe the structure of this strong exceptional collection.

In the proof we will compute mutations of exceptional collections. If~$\langle E,F\rangle$ is an exceptional collection we will denote the left mutated collection as~$\langle \mathrm{L}_EF,E\rangle$. A special property of the exceptional collection in \eqref{equation:block-collection} is that it is a~3\dash block collection, and one can also mutate blocks, for which similar notation will be used.

\begin{theorem}
  \label{theorem:exceptional-collection}
  For a generic geometric square, there is a strong exceptional collection of vector bundles
  \begin{equation}
    \langle p^* \Rscr, \mathcal{C}_0, \mathcal{C}_1, \Oscr_{\HH_{\square}} \rangle
  \end{equation}
  of ranks~$2, 2, 2, 1$ on~$\HH_\square$, where
  \begin{equation}
    \mathcal{C}_i=\ker(\Oscr_{\HH_\square}^{\oplus 2} \twoheadrightarrow \Oscr_{E_i}(1,0)).
  \end{equation}
\end{theorem}

\begin{proof}
  The first and last object are clearly vector bundles. For the middle objects, this can be checked in the fibers by tensoring the defining short exact sequence of~$\mathcal{C}_i$ with the residue field in a point, and using that~$\mathcal{O}_{E_i}(1,0)$ is the pushforward of a line bundle on~$E_i$, hence locally has the divisor short exact sequence as a flat resolution.

  The derived pullback~$\LLL p^*$ is fully faithful, and~$\LLL p^*=p^*$ when applied to vector bundles. Since~$\langle \Rscr, \Oscr_{\GG} \rangle$ is a strong exceptional pair by~\cref{theorem:decomposition-G}, so is~$\langle p^*\Rscr, \Oscr_{\HH_{\square}} \rangle$.

  We first check that~$\langle E,[F,G] \rangle=\langle \Oscr_{\HH_\square}, [\Oscr_{E_0}(1,0), \Oscr_{E_1}(1,0)] \rangle$ is a strong (block) exceptional collection. The sheaves~$\Oscr_{E_i}(1,0)$ are exceptional by the fully faithfulness of \eqref{equation:orlov-blowup-fully-faithful}; moreover
  \begin{equation}
    \begin{aligned}
      \Hom(\Oscr_{E_i}(1,0),\Oscr_{\HH_\square}[k])
      &= \Hom(\Oscr_{\HH_\square},\Oscr_{E_i}(1,0) \otimes \omega_{\HH_\square}[4-k])^\vee \\
      &= \HHH^{4-k}(\PP^1 \times \PP^2,\Oscr(-3,-2))^\vee \\
      &= 0,
    \end{aligned}
  \end{equation}
  where we used~\eqref{equation:restriction-canonical} in the second equality. Also,
  \begin{equation}
    \Hom(\Oscr_{\HH_\square},\Oscr_{E_i}(1,0)[k])=\HHH^k(\PP^1 \times\PP^2,\Oscr(1,0)),
  \end{equation}
  and~$\Oscr_{E_0}(1,0), \Oscr_{E_1}(1,0)$ are orthogonal because they have disjoint support, so~$\langle E,[F,G] \rangle$ is indeed a strong (block) exceptional collection. Hence the mutated collection
  \begin{equation}
    \langle [\mathrm{L}_E(F), \mathrm{L}_E(G)], E \rangle=\langle [\mathcal{C}_0, \mathcal{C}_1],\Oscr_{\HH_\square}\rangle
  \end{equation}
  is also exceptional. By applying~$\Hom(-,\Oscr_{\HH_\square})$ to the defining short exact sequence for~$\mathcal{C}_i$
  \begin{equation}
    \label{equation:definition-Ci}
    0 \to \mathcal{C}_i \to \Oscr^{\oplus 2}_{\HH_\square} \to \Oscr_{E_i}(1,0) \to 0,
  \end{equation}
  obtained from the mutation and using that~$\langle \Oscr_{\HH_\square}, \Oscr_{E_i}(1,0) \rangle$ is strong exceptional, we see that~$\langle [\mathcal{C}_0, \mathcal{C}_1], \Oscr_{\HH_\square} \rangle$ is a strong exceptional collection.

  It remains to check that~$\langle p^* \Rscr, [\mathcal{C}_0, \mathcal{C}_1] \rangle$ is a strong exceptional collection. We first check strongness: applying~$\Hom(p^*\Rscr,-)$ to \eqref{equation:definition-Ci} and using that~$\langle p^*(\Rscr), \Oscr_{\HH_{\square}} \rangle$ is exceptional, we find an exact sequence
  \begin{equation}
    0 \to \Hom(p^*\Rscr,\mathcal{C}_i) \to \Hom(p^*\Rscr, \Oscr_{\HH_\square}^{\oplus 2}) \to \Hom(p^*\Rscr,\Oscr_{E_i}(1,0))\to \Ext^1(p^*\Rscr,\mathcal{C}_i)\to 0,
  \end{equation}
  and
  \begin{equation}
    \Ext^{m+1}(p^* \Rscr, \mathcal{C}_i) \cong \Ext^{m}(p^*\Rscr,\Oscr_{E_i}(1,0)),
  \end{equation}
  for all~$m \geq 1$. Now
  \begin{equation}
    \Ext^{m}(p^*\Rscr,\Oscr_{E_i}(1,0)) \cong \HHH^m(\PP^1 \times \PP^2,\Oscr(2,0)^{\oplus 2}),
  \end{equation}
  which is zero for~$m\geq 1$. Also
  \begin{equation}
    \begin{aligned}
      \dim \Hom(p^*\Rscr, \Oscr_{\HH_\square}^{\oplus 2})&=8, \\
      \dim \Hom(p^*\Rscr,\Oscr_{E_i}(1,0)) &= 6,
    \end{aligned}
  \end{equation}
  so it suffices to note that
  \begin{equation}
    \begin{aligned}
      \Hom(p^*\Rscr,\mathcal{C}_i)
      &\cong\Hom_\GG(\Rscr,\mathbf{R}p_*\mathcal{C}_i) \\
      &\cong \Hom_\GG(\Rscr,\mathcal{K}_i),
    \end{aligned}
  \end{equation}
  which is~2\dash dimensional by~\cref{proposition:relations}. Finally we check exceptionality: again one can apply~$\Hom(-,p^*\Rscr)$ to~\eqref{equation:definition-Ci} to see that
  \begin{equation}
    \Ext^m(\mathcal{C}_i,p^*\Rscr) \cong \Ext^{m+1}(\Oscr_{E_i}(1,0),p^*\Rscr),
  \end{equation}
  and this last group can be calculated using Serre duality,~\cref{lemma:bundles-H-and-E} and~\cref{lemma:NLi-QLi} as follows:
  \begin{equation}
    \begin{aligned}
      \Hom(\Oscr_{E_i}(1,0),p^*\Rscr[k+1]) &\cong \Hom(p^*\Rscr, \Oscr_{E_i}(1,0)\otimes \omega_{\HH_\square}[4-k-1])^{\vee} \\
      &\cong \Hom(p^* \Rscr, \Oscr_{E_i}(-3,-2)[4-k-1])^{\vee} \\
      &\cong \HHH^{4-k-1}(\PP^1 \times \PP^2, \Oscr(-2,-2)^{\oplus 2})^{\vee},
    \end{aligned}
  \end{equation}
  which is easily seen to be zero for all~$k$.
\end{proof}

\begin{theorem}
  \label{theorem:embedding}
  For a generic geometric square~$\square$, there is an admissible embedding
  \begin{equation}
    \label{equation:main-embedding}
    \derived^{\bounded}(Q_\square) \hookrightarrow \derived^\bounded(\HH_\square),
  \end{equation}
  where~$Q_\square$ is the endomorphism algebra as in \eqref{equation:endomorphism-algebra}, and~$\HH_\square$ is a deformation of~$\HH$.
\end{theorem}

\begin{proof}
  By~\cref{theorem:exceptional-collection}, there is an admissible embedding
  \begin{equation}
    \derived^{\bounded}(\End(p^* \Rscr \oplus \mathcal{C}_0 \oplus \mathcal{C}_1 \oplus \Oscr_{\HH_{\square}})) \hookrightarrow \derived^\bounded(\HH_\square).
  \end{equation}
  Because~$i$ is a closed immersion we get the exact sequence
  \begin{equation}
    0\to\mathrm{L}^1p^*(\Oscr_{L_i}(1))\to p^*(\mathcal{K}_i)\to\Oscr_\HH^2\to\Oscr_{E_i}(1,0)\to 0
  \end{equation}
  after applying~$\mathbf{L}p^*$ to \eqref{equation:definition-Ki}  and hence by quotienting out the torsion in~$p^*(\mathcal{K}_i)$ we obtain an isomorphism
  \begin{equation}
    \mathcal{C}_i\cong p^*(\mathcal{K}_i)/\mathrm{L}^1p^*(\Oscr_{L_i}(1)).
  \end{equation}

  The last thing to observe is that the action of~$Q_\square$ remains faithful, which gives us the isomorphism
  \begin{equation}
    \End_\HH(p^* \Rscr \oplus \mathcal{C}_0 \oplus \mathcal{C}_1 \oplus \Oscr_{\HH_{\square}})\cong Q_{\square}
  \end{equation}
  and the admissible embedding \eqref{equation:main-embedding}.

  To see this, it suffices to realise that the action of~$Q_\square$ generically does not change under taking the quotient with the torsion subsheaf. So if an element of~$Q_\square$ were to act as zero on the exceptional collection on~$\HH$, it would also act as zero on the original collection of sheaves on~$\GG$ because all sheaves are torsionfree, arriving at a contradiction.
\end{proof}

\subsection{Noncommutative quadrics}
\label{subsection:noncommutative-quadrics}

We will now recall the necessary definitions and some properties of noncommutative quadrics, all of which are proven in~\cite{MR2836401}. Then we will explain how a generic noncommutative quadric gives rise to a geometric square, such that we can prove the embedding result in \cref{theorem:noncommutative-embedding}.

A~$\ZZ$\dash algebra is a pre-additive category with objects indexed by~$\ZZ$, generalising the theory of graded algebras and modules. All usual notions like right and left modules, bimodules, ideals, etc.\ make sense in this context and we will freely make use of them. For more details, consult \cite[\S2]{MR2836401}.

Let~$\Gr A$ denote the category of right~$A$-modules, and (if~$A$ is noetherian)~$\gr A$ the full subcategory of noetherian objects. Also~$\QGr A$ (respectively~$\qgr A$) is the quotient of~$\Gr A$ (respectively~$\gr A$) by the torsion modules. The quotient functor is denoted~$\pi\colon\gr A\to\qgr A$.

We write~$A_{i,j}=\Hom_A(j,i)$, and~$e_i=i\xrightarrow{\id}i$, for~$i,j \in \ZZ$. Then~$P_i=e_iA$ are projective generators for~$\Gr A$ and if~$A$ is connected,~$S_i$ will be the unique simple quotient of~$P_i$.

\begin{definition}
  A~$\ZZ$\dash algebra~$A$ is \emph{Artin--Schelter regular} if
  \begin{enumerate}
    \item $A$ is connected,
    \item $\dim A_{i,j}$ is bounded by a polynomial in~$j-i$,
    \item the projective dimension of~$S_i$ is finite, and uniformly bounded,
    \item $\sum_{j,k\in\mathbb{Z}} \dim \Ext^j_{\Gr A}(S_k,P_i)=1$, for every~$i$.
  \end{enumerate}
  If moreover the minimal resolution of~$S_i$ has the form
  \begin{equation}
    0 \to P_{i+4} \to P_{i+3}^{\oplus 2} \to P_{i+1}^{\oplus 2} \to P_i \to S_i \to 0.
  \end{equation}
  for all~$i$, then it is a \emph{three-dimensional cubic Artin--Schelter regular $\mathbb{Z}$-algebra}.
\end{definition}
Using this definition, we can now define noncommutative quadrics.
\begin{definition}
  A \emph{noncommutative quadric} is a category of the form~$\QGr A$, where~$A$ is a three-dimensional cubic Artin--Schelter regular~$\mathbb{Z}$-algebra.
\end{definition}

An important subclass of the cubic Artin--Schelter regular~$\ZZ$\dash algebras is given by the~$\ZZ$\dash algebra associated to a cubic Artin--Schelter regular graded algebra \cite{MR917738}. In general one gets a~$\ZZ$\dash algebra~$\check{B}$ from a~$\ZZ$\dash graded algebra by setting
\begin{equation}
  \check{B}_{i,j}\coloneqq B_{j-i}.
\end{equation}
The~$\mathbb{Z}$\dash algebras obtained in this way are called~1\dash periodic.

The motivation for this definition comes from the following theorem. For details and unexplained terminology we refer to~\cite{MR2836401,MR3050709}.
\begin{theorem}\cite[Theorem~1.5]{MR2836401}
  Let~$(R,\mathfrak{m})$ be a complete commutative Noetherian local ring with~$k=R/\mathfrak{m}$. Any~$R$\dash deformation of the abelian category $\coh\quadric$ is of the form~$\qgr\Ascr$, where~$\Ascr$ is an~$R$\dash family of three-dimensional cubic Artin-Schelter regular~$\ZZ$\dash algebras.
\end{theorem}

One of the main results of \cite{MR2836401} is the classification of cubic Artin--Schelter regular~$\ZZ$\dash algebras in terms of linear algebra data. We will now recall this description for use in \cref{proposition:squaremutate}.

A three-dimensional cubic AS-regular algebra satisfies~$A_{i,i+n}=0$ for~$n<0$.~It is generated by the~$V_i=A_{i,i+1}$ and the relations are generated by the
\begin{equation}
  R_i=\ker(V_i \otimes_k V_{i+1} \otimes_k V_{i+2} \to A_{i,i+3}),
\end{equation}
which are of dimension two. Denote by
\begin{equation}
  W_i=V_i \otimes_k R_{i+1} \cap R_i \otimes_k V_{i+3} \subset V_i \otimes_k V_{i+1} \otimes_k V_{i+2} \otimes_k V_{i+3},
\end{equation}
which are of dimension one. Any non-zero element of~$W_i$ is a rank two tensor, both as an element of~$V_i \otimes_k R_{i+1}$ and as an element of~$R_i \otimes_k V_{i+3}$. Finally,~$A$ is determined up to isomorphism by its truncation~$\bigoplus_{i,j=0}^3 A_{ij}$, which motivates the following definition.
\begin{definition}
  A quintuple~$(V_0, V_1, V_2, V_3,W)$, where the~$V_i$ are two-dimensional vector spaces and~$0 \neq W=kw \subset V_0 \otimes_k V_1 \otimes_k V_2 \otimes_k V_3$ is called \emph{geometric} if for all~$j \in \{0,1,2,3\}$, and for all~$0 \neq \phi_j \in V_j^\vee$, the tensor
  \begin{equation}
    \langle \phi_j \otimes_k \phi_{j+1},w \rangle
  \end{equation}
  is non-zero, where indices are taken modulo four.
\end{definition}
In the sequel we will sometimes identify a quintuple by a non-zero element of~$W$, and we will omit the tensor product.

From the previous discussion, it is clear how to associate a quintuple to a noncommutative quadric. In fact, this quintuple is geometric and there is the following classification theorem that tells us that it suffices to consider geometric quintuples.

\begin{theorem}\cite[Theorem~4.31]{MR2836401}
  There is an isomorphism preserving bijection between noncommutative quadrics and geometric quintuples.
\end{theorem}

By construction a noncommutative quadric has a full strong exceptional collection
\begin{equation}
  \begin{tikzcd}
    \pi(P_3) \arrow{r} \arrow[shift left]{r}{V_2} & \pi(P_2) \arrow{r} \arrow[shift left]{r}{V_1} & \pi(P_1) \arrow{r} \arrow[shift left]{r}{V_0} & \pi(P_0)
  \end{tikzcd}
\end{equation}
with relations~$R=W\otimes_k V_3^\vee$. We will use the (purely formal) notation
\begin{equation}
  \label{equation:sequence-collection-qgr}
  \begin{tikzcd}
    \mathcal{O}(-1,-2) \arrow{r} \arrow[shift left]{r}{V_2} & \mathcal{O}(-1,-1) \arrow{r} \arrow[shift left]{r}{V_1} & \mathcal{O}(0,-1) \arrow{r} \arrow[shift left]{r}{V_0} & \mathcal{O}(0,0).
  \end{tikzcd}
\end{equation}

\begin{example}[Linear quadric]
  \label{example:linear-quadric}
  We can now explain how the (commutative) quadric surface gives rise to a cubic Artin--Schelter regular~$\ZZ$\dash algebra. On~$\quadric$ there are the line bundles
  \begin{equation}
    \Oscr_{\quadric}(m,n)=\Oscr_{\PP^1}(m) \boxtimes \Oscr_{\PP^1}(n).
  \end{equation}
  The following defines an ample sequence:
  \begin{equation}
    \Oscr_{\quadric}(n)=
    \begin{cases}
      \Oscr_{\quadric}(k,k) & \text{ if } n=2k,\\
      \Oscr_{\quadric}(k+1,k) & \text{ if } n=2k+1.
    \end{cases}
  \end{equation}
  Put~$A=\bigoplus_{i,j} \Hom_{\quadric}(\Oscr_{\quadric}(-j),\Oscr_{\quadric}(-i))$. Then~$\coh\quadric \cong \qgr A$, and~$A$ is a~$3$\dash dimensional cubic~AS\dash regular algebra. One may choose bases~$x_i, y_i$ for~$V_i$ such that the relations in~$A$ are given by
  \begin{equation}
    \label{equation:linear-quadric-relations}
    \begin{aligned}
      x_i x_{i+1} y_{i+2} - y_i x_{i+1} x_{i+2} &= 0\\
      x_i y_{i+1} y_{i+2} - y_i y_{i+1} x_{i+2} &= 0.
    \end{aligned}
  \end{equation}
  The tensor~$w \in W_0$ corresponding to these relations is given by
  \begin{equation}
    \label{equation:linear-quadric}
    w=x_0x_1y_2y_3 - y_0x_1x_2y_3-x_0y_1y_2x_3+y_0y_1x_2x_3.
  \end{equation}
  The corresponding exceptional collection has quiver
  \begin{equation}
    \label{equation:linear-collection-qgr}
    \begin{tikzcd}
      \Oscr_{\quadric}(-3) \arrow[swap]{r}{y_2} \arrow[shift left]{r}{x_2} & \Oscr_{\quadric}(-2) \arrow[swap]{r}{y_1} \arrow[shift left]{r}{x_1} & \Oscr_{\quadric}(-1) \arrow[swap]{r}{y_0} \arrow[shift left]{r}{x_0} & \Oscr_{\quadric}
    \end{tikzcd}
  \end{equation}
  with relations~\eqref{equation:linear-quadric-relations}, corresponding to~\eqref{equation:sequence-collection-qgr}.

  The relationship between the homogeneous coordinate ring of~$\quadric$ under the Segre embedding and the~$\ZZ$\dash algebra~$A$ is obtained by taking the~$2$\dash Veronese of~$A$, giving an isomorphism
  \begin{equation}
    \Bigg( k\langle x,y \rangle / \begin{pmatrix} x^2y - yx^2 \\ xy^2-y^2x \end{pmatrix} \Bigg)_2 \cong k[a,b,c,d]/(ad-bc),
  \end{equation}
  where we descibed the~$\ZZ$\dash algebra as a graded algebra, because in this case~$A$ is~1\dash periodic.
\end{example}

Another important class of noncommutative quadrics is given by the so called type-A cubic algebras.
\begin{example}[Type-A cubic algebras]
  \label{example:type-a}
  We will consider the generic class of cubic algebras from \cite{MR917738}. In this case the (graded) algebra~$A$ has two generators~$x$ and~$y$ and relations
  \begin{equation}
    \label{equation:type-a-relations}
    \begin{aligned}
      ay^2x+byxy+axy^2+cx^3&=0 \\
      ax^2y+bxyx+ayx^2+cy^3&=0
    \end{aligned}.
  \end{equation}
  These algebras are Artin--Schelter regular for~$(a:b:c) \in \PP^2-S$, where
  \begin{equation}
    S=\{(a:b:c) \in \PP^2\mid a^2=b^2=c^2\} \cup \{(0:0:1),(0:1:0)\},
  \end{equation}

  The tensor~$w \in W_0$ corresponding to these relations in the~$\ZZ$\dash algebra setting is given by
  \begin{equation}
    \begin{aligned}
      w&=ay_0y_1x_2x_3+by_0x_1y_2x_3+ax_0y_1y_2x_3+cx_0x_1x_2x_3 \\
      &\quad+ax_0x_1y_2y_3+bx_0y_1x_2y_3+ay_0x_1x_2y_3+cy_0y_1y_2y_3.
    \end{aligned}
  \end{equation}

  The corresponding full and strong exceptional collection is given by
  \begin{equation}
    \begin{tikzcd}
      A(-3) \arrow[swap]{r}{y_2} \arrow[shift left]{r}{x_2} & A(-2) \arrow[swap]{r}{y_1} \arrow[shift left]{r}{x_1} & A(-1) \arrow[swap]{r}{y_0} \arrow[shift left]{r}{x_0} & A
    \end{tikzcd}
  \end{equation}
  with relations coming from~\eqref{equation:type-a-relations}.
\end{example}

Since our model for \cref{theorem:embedding} was the 3-block exceptional collection~\eqref{equation:block-collection} and not the linear collection~\eqref{equation:linear-collection-qgr}, we first have to mutate a linear exceptional collection as in \eqref{equation:sequence-collection-qgr} to a square one as in \eqref{equation:block-collection}.
\begin{proposition}
  \label{proposition:squaremutate}
  The exceptional collection obtained from~\eqref{equation:sequence-collection-qgr} by right mutating the first two objects is strong and has endomorphism ring
  \begin{equation}
    \label{equation:block-collection-qgr}
    \begin{tikzcd}
      & \Oscr(-1,0) \arrow[swap]{dr} \arrow[shift left = .6em]{dr}{R} \\
      \Oscr(-1,-1) \arrow[swap]{ur} \arrow[shift left = .6em]{ur}{V_2^\vee} \arrow[swap]{dr}{V_1} \arrow[shift left = .6em]{dr} & & \Oscr(0,0) \\
      & \Oscr(0,-1) \arrow[swap]{ur}{V_0} \arrow[shift left = .6em]{ur}
    \end{tikzcd}
  \end{equation}
  where we used the notation~$\Oscr(-1,0)=\mathrm{R}_{\Oscr(-1,-1)}\Oscr(-1,-2)$.
\end{proposition}

\begin{proof}
  By construction the right mutation~$\Oscr(-1,0)$ fits in a short exact sequence
  \begin{equation}
    \label{equation:right-mutation}
    0\to\mathcal{O}(-1,-2)\to V_2^\vee\otimes_k\mathcal{O}(-1,-1)\to\mathcal{O}(-1,0)\to 0
  \end{equation}
  because we can compute the mutation entirely in~$\qgr A$ as the morphism on the left is indeed a monomorphism by definition.

  To see that~$\Hom(\Oscr(-1,0),\Oscr(0,0))=R$ one can use the proof of \cite[Lemma~4.3]{MR2836401}. By applying~$\Hom(-,\mathcal{O}(0,0))$ to \eqref{equation:right-mutation} we get a long exact sequence, which by the canonical isomorphism~$A_{0,2}=V_0\otimes V_1=\Hom(\Oscr(-1,-1),\Oscr(0,0))$ corresponds to
  \begin{equation}
    0\to\Hom(\mathcal{O}(-1,0),\mathcal{O}(0,0))\to V_0\otimes_k V_1\otimes_k V_2\to A_{0,3}\to 0,
  \end{equation}
  hence~$R=\Hom(\mathcal{O}(-1,0),\mathcal{O}(0,0)$. This also shows that the higher Ext's vanish.

  Finally, to see that~$\Oscr(-1,0)$ and~$\Oscr(0,-1)$ are completely orthogonal, we can apply~$\Hom(-,\mathcal{O}(0,-1))$ to \eqref{equation:right-mutation}. By the resulting long exact sequence where the isomorphism~$V_1\otimes_k V_2\cong A_{1,3}$ is the only non-zero map we get the desired orthogonality.
\end{proof}

We are almost in a situation where we can apply \cref{theorem:embedding}. However, an arbitrary geometric quintuple does not give rise to a geometric square since the induced map~$R \otimes_k V_2^{\vee} \to V_0 \otimes_k V_1$ in~\eqref{equation:block-collection-qgr} is not necessarily an isomorphism. The next proposition describes a dense subset for which this is the case. Recall that~$w\in V_0\otimes_k V_1\otimes_k V_2\otimes_k V_3$, and we have an action of~$\mathbb{G}_{\mathrm{m}}$ on this space, so~$w$ can be interpreted as a point in~$\mathbb{P}^{15}$.
\begin{proposition}
  \label{proposition:linear-to-block}
  A generic geometric quintuple~$(V_0,V_1,V_2,V_3,w)$ gives rise to a geometric square. More precisely, for~$w$ in a Zariski open subset~$\Uscr'$ of~$\PP^{15}$,
  \begin{equation}
    \label{equation:associated-square}
    \square_w=(V_0 \otimes_k V_1, V_0, V_1, V_2^\vee, V_3^\vee, \id, \phi_w)
  \end{equation}
  is a geometric square, where~$\phi_w=\langle -,w\rangle^{-1}$.
\end{proposition}

\begin{proof}
  The condition that the morphism
  \begin{equation}
    \langle-,w\rangle\colon V_2^\vee\otimes_k V_{3}^\vee\to V_{0}\otimes_k V_{1}
  \end{equation}
  induced by an element~$w\in V_0\otimes_k V_1\otimes_k V_2\otimes_k V_3$ is an isomorphism is given by the non-vanishing of the determinant. The open subset~$\Uscr'$ is defined as the intersection of the locus of geometric quintuples with the complement of this vanishing locus in~$\PP^{15}$. So starting from a geometric quintuple with~$w \in \Uscr'$ we can define the associated square~\eqref{equation:associated-square}.
\end{proof}

\begin{remark}
  Remark that the condition required for \cref{proposition:linear-to-block} is indeed stronger than the geometricity condition for a quintuple. This geometricity condition ensures that the morphism~$\langle -,w\rangle$ sends the pure tensors to nonzero elements. This does not imply that the morphism is an isomorphism, only that the kernel has to intersect the quadric cone corresponding to the pure tensors trivially in the origin. This implies that the kernel is necessarily of dimension~1.
\end{remark}

Let us denote by~$\HH_w\coloneqq\HH_{\square_w}$, for~$w \in \Uscr'$, and by~$\qgr A_w$ the associated noncommutative quadric. The following is then our main result.

\begin{theorem}
  \label{theorem:noncommutative-embedding}
  The varieties~$\HH_{w}$ form a smooth projective family~$\Hscr$ over a Zariski open~$\Uscr \subset \Uscr'$ containing~$\HH$, and for each~$w\in\Uscr$ there is an admissible embedding
  \begin{equation}
    \derived^\bounded(\qgr A_w) \hookrightarrow \derived^\bounded(\HH_w).
  \end{equation}
  by vector bundles of ranks~$2, 2, 2$ and~$1$.
\end{theorem}

\begin{proof}
  This is now immediate from the combination of~\cref{theorem:embedding},~\cref{proposition:linear-to-block} and~\cref{proposition:squaremutate}. Note that we have to restrict to a Zariski open~$\Uscr \subset \Uscr'$ since~\cref{theorem:embedding} only works for a generic geometric square for which the corresponding~$\PP^1$'s do not intersect. Also,~$\HH$ is a member of the family by~\cref{example:linear-final}.
\end{proof}

\begin{example}[Linear quadric]
  \label{example:linear-final}
  For the geometric quintuple \eqref{equation:linear-quadric} it is easy to see that~$w \in \Uscr$, so we get an associated geometric square with exceptional collection~\eqref{equation:block-collection-qgr}, which is exactly the 3-block collection~\eqref{equation:block-collection}. Another small calculation shows that the two~$\PP^1$'s don't intersect so~$w \in \Uscr'$. As expected, the two~$\PP^1$'s correspond to the two rulings on~$\PP^1 \times \PP^1$, which we used in~\cref{theorem:description-H}.

\end{example}

\begin{example}
  Consider the type-A cubic algebra from \cref{example:type-a} for the parameters~$(0:1:1)$. In this case the matrix describing~$\phi_w$ is the identity matrix, hence the two~$\mathbb{P}^1$'s coincide and~\cref{theorem:embedding} does not apply.
\end{example}

\section{Further remarks}
\label{section:remarks}

Based on the result for~$\PP^2$ from \cite{zbMATH06527075} Orlov conjectured informally that every noncommutative deformation can be embedded in some commutative deformation, i.e.\ for every smooth projective variety~$X$ there exists a smooth projective variety~$Y$ and a fully faithful functor~$\derived^\bounded(X)\hookrightarrow\derived^\bounded(Y)$ such that for every noncommutative deformation of~$X$ there is a commutative deformation of~$Y$ such that there is again a fully faithful functor between the bounded derived categories.

The result in this paper adds some further evidence to this, by proving the result for~$X=\quadric$ and~$Y=\Hilb^2\quadric$. The general construction from \cite{zbMATH06527075} seems to prove this conjecture in case~$\derived^\bounded(X)$ has a full and strong exceptional collection: noncommutative deformations of~$X$ correspond to changing the relations in the quiver, and these changes are reflected by changing the vector bundles in the iterated projective bundle construction.

However, it would be interesting to know whether one can always choose for~$Y$ a natural moduli space associated to~$X$, as is the case for~$\PP^2$ and~$\PP^1 \times \PP^1$ where one can take the Hilbert scheme of two points. To investigate this in  a more general setting we formulate an infinitesimal version of this conjecture in terms of limited functoriality for Hochschild cohomology, and explain how results on Poisson structures on surfaces give some substance to this conjecture in special cases.

The infinitesimal deformation theory of abelian categories is governed by their Hochschild cohomology \cite{MR2183254}, and one has the Hochschild--Kostant--Rosenberg decomposition for Hochschild cohomology of smooth varieties. In particular there is the decomposition
\begin{equation}
  \HHHH^2(X)=\HHH^0(X,\bigwedge^2\Tscr_X)\oplus\HHH^1(X,\Tscr_X)\oplus\HHH^2(X,\mathcal{O}_X)
\end{equation}
where the first term can be understood as the noncommutative deformations, the second as the commutative (or geometric) deformations and the third one corresponding to gerby deformations \cite{MR2477894}.

The natural categorical framework for Hochschild cohomology is that of dg~categories. It is easily checked that Hochschild cohomology is not functorial for arbitrary functors: it only satisfies a limited functoriality. Indeed, in the case of a dg~functor inducing a fully faithful embedding on the level of derived categories there is an induced morphism on the Hochschild cohomologies \cite{keller-derived-invariance}, which in the case of Fourier--Mukai transforms is treated in \cite{MR3420334}.

Combining limited functoriality with the Hochschild--Kostant--Rosenberg decomposition one could formulate an infinitesimal version of Orlov's conjecture as follows.
\begin{question}
  Let~$X$ be a smooth projective variety. Does there exist a smooth projective variety~$Y$ and a fully faithful embedding~$\derived^\bounded(X)\to\derived^\bounded(Y)$, such that the induced morphism on Hochschild cohomologies induces a surjective morphism
  \begin{equation}
    \label{equation:conjecture-morphism}
    \HHH^1(Y,\Tscr_Y)\twoheadrightarrow\HHH^0(X,\bigwedge^2\Tscr_X).
  \end{equation}
\end{question}
Sadly, we do not even know the answer for the embeddings obtained for~$X=\PP^2$ or~$\quadric$ and~$Y=\Hilb^2X$.

Some positive evidence comes from a result by Hitchin who shows in \cite{MR3024823} the existence of the split exact sequence
\begin{equation}
  \label{equation:hitchin-sequence}
  0\to\HHH^1(S,\Tscr_S)\to\HHH^1(\Hilb^nS,\Tscr_{\Hilb^nS})\to\HHH^0(S,\omega_S^\vee)\to 0
\end{equation}
where~$S$ is a smooth projective surface over the complex numbers.

Again one does not know that the morphism on the right is related to \eqref{equation:conjecture-morphism}, but it does show that a possible approach might be to choose for~$Y$ a smooth projective variety representing a moduli problem associated to~$X$.

\begin{remark}
  The choice of the Hilbert scheme of~$n$ points seems to be a natural choice in the case of a surface with exceptional structure sheaf, but for higher-dimensional varieties the Hilbert scheme fails to be smooth in general. For~$n=2$ they are smooth though, but in \cite{bfr} we generalise \eqref{equation:krug-sosna-embedding} to the case of Hilbert squares of higher-dimensional varieties with exceptional structure sheaf, and use it to show that~$\HHH^1(\Hilb^2X,\Tscr_{\Hilb^2X})\cong\HHH^1(X,\Tscr_X)$, i.e.~there is no contribution of the noncommutative deformations of~$X$.
\end{remark}

Closely related to the suggested question is a correspondence between the Hochschild cohomology of a noncommutative plane and the deformation of the Hilbert scheme of~2~points on~$\mathbb{P}^2$ whose derived category contains the derived category of the noncommutative plane. At least for a Sklyanin algebra the Hochschild cohomology as computed in \cite{1705.06098} agrees with the commutative deformations in the Hochschild--Kostant--Rosenberg decomposition of the deformed Hilbert scheme, by \cite[proposition~11]{MR3024823}. Understanding this phenomenon in greater detail is work in progress.

\appendix
\section{A proof of Lemma \ref{lemma:factorization}}
In this appendix we give a detailed proof of Lemma \ref{lemma:factorization}. Let us recall the setup and notation. Given a vector space~$V$ of dimension~4 we let~$\GG$ denote the Grassmannian of~$2$\dash dimensional quotients of~$V$. Let
\begin{equation}
  \label{equation:tautological-sequence-appendix}
  0 \to \Rscr \xrightarrow{r} V\otimes_k \Oscr_\GG \xrightarrow{q} \Qscr \to 0
\end{equation}
be the tautological exact sequence on~$\GG$, where~$\Qscr$ is the universal quotient bundle of rank~$2$, and~$\Rscr$ is the universal subbundle of rank~$2$. We have that~$V=\Hom_\GG(\Oscr_\GG,\Qscr)$, and~$q$ is the evaluation morphism.

Moreover,~$V_0, V_1$ denote vector spaces of dimension~$2$, and~$\chi\colon V \to V_0 \otimes_k V_1$ is a given isomorphism. On~$\PP\coloneqq\PP(V_0)$ we have the canonical quotient morphism
\begin{equation}
  V_0\otimes_k\Oscr_\PP\overset{\pi}{\to}\Oscr_\PP(1)
\end{equation}
which gives us the morphism
\begin{equation}
  \label{equation:u}
  u\colon V\otimes_k\Oscr_\PP\xrightarrow{\chi\otimes\identity_{\Oscr_\PP}}V_1\otimes_kV_0\otimes_k\Oscr_\PP\xrightarrow{\identity_{V_1}\otimes\pi}V_1\otimes_k\Oscr_\PP(1).
\end{equation}
From this we obtain the classifying morphism~$\Phi=\Phi_\chi\colon\PP\to\GG$, such that~$\Phi^*q\simeq u$. The object~$\mathcal{K}_\chi$ is defined via the short exact sequence
\begin{equation}
  \label{equation:definition-K-chi}
  0\to\mathcal{K}_\chi\to\HHH^0(\GG,\Phi_*\Oscr_\PP(1))\otimes_k\Oscr_\GG\overset{\eval}{\to}\Phi_*\Oscr_\PP(1)\to 0.
\end{equation}

We will use the following shorthand
\begin{equation}
  A\coloneqq\Oscr_\GG,
  \qquad B\coloneqq\mathcal{Q},
  \qquad C\coloneqq\Phi_*\Oscr_\PP(1).
\end{equation}

\begin{lemma}
\label{lem:mutation}
Seen as distinguished triangles, the sequences~\eqref{equation:tautological-sequence-appendix} and~\eqref{equation:definition-K-chi} respectively coincide with the mutation triangles
\begin{equation}
  \label{equation:left-1}
  B[-1]\xrightarrow[+1]{b}\LL_AB\xrightarrow{\beta}\RHom(A,B)\otimes_kA\xrightarrow{\eval} B
\end{equation}
and
\begin{equation}
  \label{equation:left-2}
  C[-1]\xrightarrow[+1]{c}\LL_AC\xrightarrow{\gamma}\RHom(A,C)\otimes_kA\xrightarrow{\eval} C.
\end{equation}
\end{lemma}
\begin{proof}
For~\eqref{equation:tautological-sequence-appendix}, this is standard, and follows for example from the construction of the dual exceptional collection from Theorem \ref{theorem:decomposition-G}. For~\eqref{equation:definition-K-chi}, this follows since $\Ext^i_{\GG}(\Oscr_\GG,\Phi_*\Oscr_\PP(1))=\HHH^i(\PP,\Oscr_\PP(1))=0$, for $i>0$.
\end{proof}

Note however that the object~$C$ is not exceptional. Now Lemma \ref{lemma:factorization} can be restated as follows.
\begin{lemma}
\label{lem:reformulated}
  The composition morphism
  \begin{equation}
    \label{equation:composition-morphism}
    \Hom(\LL_AB,\LL_AC)\otimes_k\Hom(\LL_AC,A)\to\Hom(\LL_AB,A)
  \end{equation}
  is an isomorphism.
\end{lemma}
The proof of Lemma \ref{lem:reformulated} now follows by combining Lemma \ref{lem:numbers} and Lemma \ref{lem:chi} below. The idea is to first identify \eqref{equation:composition-morphism} with
\begin{equation}
  \label{equation:ABC-evaluation}
  \Hom(B,C)\otimes_k\Hom(A,C)^\vee\to\Hom(A,B)^\vee:f\otimes g^\vee\mapsto(h\mapsto\langle g^\vee,f\circ h\rangle)
\end{equation}
and then to identify the~$k$\dash linear dual
\begin{equation}
  \label{equation:k-dual-ABC}
  \Hom(A,B)\to\Hom(A,C)\otimes_k\Hom(B,C)^\vee
\end{equation}
of \eqref{equation:ABC-evaluation} with the isomorphism~$\chi$.

For the first step we will use the following sequences of canonical isomorphisms:
\begin{equation}
  \label{equation:iso-1}
  \Hom(B,C)
  \xrightarrow[\simeq]{[-1]}
  \Hom(B[-1],C[-1])
  \xrightarrow[\simeq]{c_*}
  \Hom(B[-1],\LL_AC)
  \xrightarrow[\simeq]{(b_*)^{-1}}
  \Hom(\LL_AB,\LL_AC),
\end{equation}
\begin{equation}
  \label{equation:iso-2}
  \Hom(A,B)^\vee
  \xrightarrow[\simeq]{}
  \Hom\left( \Hom(A,B)\otimes_kA,A \right)
  \xrightarrow[\simeq]{\beta^*}
  \Hom(\LL_AB,A)
\end{equation}
and
\begin{equation}
  \label{equation:iso-3}
  \Hom(A,C)^\vee
  \xrightarrow[\simeq]{}
  \Hom\left( \Hom(A,C)\otimes_kA,A \right)
  \xrightarrow[\simeq]{\gamma^*}
  \Hom(\LL_AC,A).
\end{equation}
The chain of isomorphisms in \eqref{equation:iso-1} uses that~$\langle A,B \rangle$ is an exceptional pair, and that~$\Hom(A,\LL_AC[i])=0$ for all~$i$.
The chain of isomorphisms in \eqref{equation:iso-2} follows from~$\langle A,B\rangle$ being an exceptional pair.
The final chain of isomorphisms in \eqref{equation:iso-3} follows from~$A$ being exceptional and~$\Hom(C,A[i])=0$ for~$i=0,1$.

\begin{lemma}
\label{lem:numbers}
  The composition morphism \eqref{equation:composition-morphism} coincides with \eqref{equation:ABC-evaluation}, using the isomorphisms \eqref{equation:iso-1}, \eqref{equation:iso-2} and \eqref{equation:iso-3}.
\end{lemma}

\begin{proof}
  Consider~$f\in\Hom(B,C)$ and~$g^\vee\in\Hom(A,C)^\vee$. The triangles \eqref{equation:left-1} and \eqref{equation:left-2} give rise to the commutative diagram
  \begin{equation}
    \begin{tikzcd}
      B[-1] \arrow[r, "b"] \arrow[d, "{f[-1]}"] & \LL_AB \arrow[r, "\beta"] \arrow[d, dashed, "\widetilde{f}"] & \Hom(A,B)\otimes_kA \arrow[d, "f_*\otimes\identity_A"] \\
      C[-1] \arrow[r, "c"] & \LL_AC \arrow[r, "\gamma"] & \Hom(A,C)\otimes_kA  \\
    \end{tikzcd}
  \end{equation}
The assertion now follows by considering $\widetilde{g}\coloneqq(g^\vee \otimes \id_A)\circ \gamma\colon \LL_AC \to A$, and using that commutativity ensures that~$\widetilde{g} \circ \widetilde{f}=(g^\vee \otimes \id_A) \circ (f_* \otimes \id_A) \circ \beta$ combined with \eqref{equation:iso-1}, \eqref{equation:iso-2} and \eqref{equation:iso-3}.
\end{proof}

The~$k$\dash linear dual of \eqref{equation:ABC-evaluation} coincides with the map
\begin{equation}
\label{equation:coev}
  \Hom(A,B)\to\Hom(A,C)\otimes_k\Hom(B,C)^\vee
\end{equation}
obtained by applying $\Hom(A,-)$ to the coevaluation map $B \to \Hom(B,C)^\vee \otimes_k C$. We also have the following sequences of canonical isomorphisms:
\begin{equation}
  \label{equation:iso-2-1}
  V
  \xrightarrow[\simeq]{}
  \Hom_\GG(\Oscr_\GG,V\otimes_k\Oscr_\GG)
  \xrightarrow[\simeq]{}
  \Hom_\GG(\Oscr_\GG,\mathcal{Q})
  =
  \Hom(A,B),
\end{equation}
\begin{equation}
  \label{equation:iso-2-2}
  V_1^\vee
  \xrightarrow[\simeq]{}
  \Hom_\PP(V_1\otimes_k\Oscr_\PP(1),\Oscr_\PP(1))
  \xrightarrow[\simeq]{}
  \Hom_\PP(\Phi^*\mathcal{Q},\Oscr_\PP(1))
  \xrightarrow[\simeq]{}
  \Hom_\GG(\mathcal{Q},\Phi_*\Oscr_\PP(1))
  =\Hom(B,C)
\end{equation}
and
\begin{equation}
  \label{equation:iso-2-3}
  V_0
  \xrightarrow[\simeq]{}
  \Hom_\PP(\Oscr_\PP,\Oscr_\PP(1))
  \xrightarrow[\simeq]{}
  \Hom_\PP(\Phi^*\Oscr_\GG,\Oscr_\PP(1))
  \xrightarrow[\simeq]{}
  \Hom_\GG(\Oscr_\GG,\Phi_*\Oscr_\PP(1))
  =
  \Hom(A,C).
\end{equation}
The second isomorphism in \eqref{equation:iso-2-2} was discussed in \eqref{equation:u} (and contains the information about~$\chi$). The other ones are all standard.

\begin{lemma}
\label{lem:chi}
  The morphism~$\chi$ coincides with \eqref{equation:coev} using the isomorphisms \eqref{equation:iso-2-1}, \eqref{equation:iso-2-2} and \eqref{equation:iso-2-3}.
\end{lemma}

\begin{proof}
  By adjunction the statement of the lemma is equivalent to the claim that the composition morphism
  \begin{equation}
  \label{equation:composition-again}
    \Hom(A,B)\otimes_k\Hom(B,C)\to\Hom(A,C)
  \end{equation}
  coincides with the morphism
  \begin{equation}
    V\otimes_kV_1^\vee\to V_0:v\otimes v_1^\vee\mapsto\langle v_1^\vee,\chi(v) \rangle,
  \end{equation}
  where we use \eqref{equation:iso-2-1}, \eqref{equation:iso-2-2} and \eqref{equation:iso-2-3} to identify the spaces.

  To see this, we consider the diagram
  \begin{equation}
    \label{equation:commutative-diagram}
    \begin{tikzcd}
      \Hom_\GG(\Oscr_\GG,\mathcal{Q})\otimes_k\Hom_\GG(\mathcal{Q},\Phi_*\Oscr_\PP(1)) \arrow[r, "-\circ-"] \arrow[d, "\simeq"] & \Hom_\GG(\Oscr_\GG,\Phi_*\Oscr_\PP(1)) \arrow[dd, "\simeq"] \\
      \Hom_\GG(\Oscr_\GG,\mathcal{Q})\otimes_k\Hom_\PP(\Phi^*\mathcal{Q},\Oscr_\PP(1)) \arrow[d, "\Phi^*\otimes\identity"] \\
      \Hom_\PP(\Phi^*\Oscr_\GG,\Phi^*\mathcal{Q})\otimes_k\Hom_\PP(\Phi^*\mathcal{Q},\Oscr_\PP(1)) \arrow[r, "-\circ-"] & \Hom_\PP(\Phi^*\Oscr_\GG,\Oscr_\PP(1)).
    \end{tikzcd}
  \end{equation}
  which commutes since~$\Phi^*$ and~$\Phi_*$ are adjoint functors.

  Letting~$f\in\Hom(A,B)$ and~$g\in\Hom(B,C)$ correspond to~$v\in V$ and~$v_1^\vee\in V_1^\vee$ respectively we obtain by the commutativity of \eqref{equation:commutative-diagram} that
  \begin{equation}
    \operatorname{adjoint}(g\circ f)=\operatorname{adjoint}(g)\circ\Phi^*f,
  \end{equation}
  where we explicitly indicate the adjunction isomorphisms.

  It follows that the morphism~$V\otimes_kV_1^\vee\to V_0\cong\Hom_\PP(\Oscr_\PP,\Oscr_\PP(1))$ (induced by \eqref{equation:composition-again} and \eqref{equation:iso-2-1}, \eqref{equation:iso-2-2} and \eqref{equation:iso-2-3}) sends~$v\otimes v_1^\vee$ to the morphism
  \begin{equation}
    (v_1^\vee\otimes\identity_{\Oscr_\PP(1)})\circ u\circ (v\otimes\identity_{\Oscr_\PP})\colon\Oscr_\PP\to\Oscr_\PP(1).
  \end{equation}
  But this proves the lemma, as we can rewrite this as
  \begin{equation}
    (v_1^\vee\otimes\identity_{\Oscr_\PP(1)})\circ(\identity_{V_1}\otimes\pi)\circ(\chi(v)\otimes\identity_{\Oscr_\PP})=\pi\circ(\langle v_1^\vee,\chi(v)\rangle\otimes\identity_{\Oscr_\PP}).
  \end{equation}
\end{proof}

\bibliography{mybibstheo-v5}{}

\providecommand{\bysame}{\leavevmode\hbox to3em{\hrulefill}\thinspace}
\providecommand{\MR}{\relax\ifhmode\unskip\space\fi MR }
\providecommand{\MRhref}[2]{%
  \href{http://www.ams.org/mathscinet-getitem?mr=#1}{#2}
}
\providecommand{\href}[2]{#2}
\begin{thebibliography}{10}

\bibitem{MR917738}
Michael Artin and William~F. Schelter, \emph{Graded algebras of global
  dimension {$3$}}, Adv. in Math. \textbf{66} (1987), no.~2, 171--216.
  \MR{917738 (88k:16003)}

\bibitem{1705.06098}
Pieter Belmans, \emph{Hochschild cohomology of noncommutative planes and
  quadrics}, 2017, arXiv:1705.06098.

\bibitem{bfr}
Pieter Belmans, Lie Fu, and Theo Raedschelders, \emph{Derived categories and
  deformations of {H}ilbert squares}, 2018, work in progress.

\bibitem{MR3371493}
Ragnar-Olaf Buchweitz, Graham~J. Leuschke, and Michel Van~den Bergh, \emph{On
  the derived category of {G}rassmannians in arbitrary characteristic}, Compos.
  Math. \textbf{151} (2015), no.~7, 1242--1264. \MR{3371493}

\bibitem{MR1065935}
Fabrizio Catanese and Lothar G\"ottsche, \emph{{$d$}-very-ample line bundles
  and embeddings of {H}ilbert schemes of {$0$}-cycles}, Manuscripta Math.
  \textbf{68} (1990), no.~3, 337--341. \MR{1065935}

\bibitem{MR2309895}
Koen De~Naeghel and Nicolas Marconnet, \emph{Ideals of cubic algebras and an
  invariant ring of the {W}eyl algebra}, J. Algebra \textbf{311} (2007), no.~1,
  380--433. \MR{2309895}

\bibitem{MR2584227}
A.~D. Elagin, \emph{Semi-orthogonal decompositions for derived categories of
  equivariant coherent sheaves}, Izv. Ross. Akad. Nauk Ser. Mat. \textbf{73}
  (2009), no.~5, 37--66. \MR{2584227}

\bibitem{MR0237496}
John Fogarty, \emph{Algebraic families on an algebraic surface}, Amer. J. Math
  \textbf{90} (1968), 511--521. \MR{0237496}

\bibitem{MR1288523}
Phillip Griffiths and Joseph Harris, \emph{Principles of algebraic geometry},
  Wiley Classics Library, John Wiley \& Sons, Inc., New York, 1994, Reprint of
  the 1978 original. \MR{1288523 (95d:14001)}

\bibitem{MR3024823}
Nigel Hitchin, \emph{Deformations of holomorphic {P}oisson manifolds}, Mosc.
  Math. J. \textbf{12} (2012), no.~3, 567--591, 669. \MR{3024823}

\bibitem{MR939472}
M.~M. Kapranov, \emph{On the derived categories of coherent sheaves on some
  homogeneous spaces}, Invent. Math. \textbf{92} (1988), no.~3, 479--508.
  \MR{939472 (89g:18018)}

\bibitem{keller-derived-invariance}
Bernhard Keller, \emph{Derived invariance of higher structures on the
  {H}ochschild complex}, 2003.

\bibitem{MR3397451}
Andreas Krug and Pawel Sosna, \emph{On the derived category of the {H}ilbert
  scheme of points on an {E}nriques surface}, Selecta Math. (N.S.) \textbf{21}
  (2015), no.~4, 1339--1360. \MR{3397451}

\bibitem{MR3420334}
Alexander Kuznetsov, \emph{Height of exceptional collections and {H}ochschild
  cohomology of quasiphantom categories}, Journal f{\"u}r die reine und
  angewandte Mathematik (Crelles Journal) \textbf{2015} (2015), no.~708,
  213--243.

\bibitem{MR2183254}
Wendy Lowen and Michel Van~den Bergh, \emph{Hochschild cohomology of abelian
  categories and ringed spaces}, Adv. Math. \textbf{198} (2005), no.~1,
  172--221. \MR{2183254 (2007d:18017)}

\bibitem{MR2238922}
\bysame, \emph{Deformation theory of abelian categories}, Trans. Amer. Math.
  Soc. \textbf{358} (2006), no.~12, 5441--5483 (electronic). \MR{2238922
  (2008b:18016)}

\bibitem{okawa-ueda-nc-quadrics}
Shinnosuke Okawa and Kazushi Ueda, \emph{Noncommutative quadric surfaces and
  noncommutative conifolds}, 2014, arXiv:1403.0713.

\bibitem{MR1208153}
D.~O. Orlov, \emph{Projective bundles, monoidal transformations, and derived
  categories of coherent sheaves}, Izv. Ross. Akad. Nauk Ser. Mat. \textbf{56}
  (1992), no.~4, 852--862. \MR{1208153 (94e:14024)}

\bibitem{zbMATH06527075}
Dmitri Orlov, \emph{Geometric realizations of quiver algebras}, Proc. Steklov
  Inst. Math. \textbf{290} (2015), no.~1, 70--83. \MR{3488782}

\bibitem{orlov-smooth-and-proper}
\bysame, \emph{Smooth and proper noncommutative schemes and gluing of {DG}
  categories}, Adv. Math. \textbf{302} (2016), 59--105. \MR{3545926}

\bibitem{MR2327036}
Dalide Pontoni, \emph{Quantum cohomology of {${\rm Hilb}^2(\Bbb P^1\times\Bbb
  P^1)$} and enumerative applications}, Trans. Amer. Math. Soc. \textbf{359}
  (2007), no.~11, 5419--5448. \MR{2327036}

\bibitem{MR3108697}
S.~Paul Smith and Michel Van~den Bergh, \emph{Noncommutative quadric surfaces},
  J. Noncommut. Geom. \textbf{7} (2013), no.~3, 817--856. \MR{3108697}

\bibitem{MR2477894}
Yukinobu Toda, \emph{Deformations and {F}ourier-{M}ukai transforms}, J.
  Differential Geom. \textbf{81} (2009), no.~1, 197--224. \MR{2477894
  (2010a:14020)}

\bibitem{MR2836401}
Michel Van~den Bergh, \emph{Noncommutative quadrics}, Int. Math. Res. Not. IMRN
  (2011), no.~17, 3983--4026. \MR{2836401 (2012m:14004)}

\bibitem{MR3050709}
\bysame, \emph{Notes on formal deformations of abelian categories}, Derived
  categories in algebraic geometry, EMS Ser. Congr. Rep., Eur. Math. Soc.,
  Z{\"u}rich, 2012, pp.~319--344. \MR{3050709}

\end{thebibliography}
\bibliographystyle{amsplain}
\def\cprime{$'$} \def\cprime{$'$} \def\cprime{$'$}
\providecommand{\bysame}{\leavevmode\hbox to3em{\hrulefill}\thinspace}
\providecommand{\MR}{\relax\ifhmode\unskip\space\fi MR }
\providecommand{\MRhref}[2]{%
  \href{http://www.ams.org/mathscinet-getitem?mr=#1}{#2}
}
\providecommand{\href}[2]{#2}

\end{document}